\documentclass[a5paper]{book}
\usepackage[english,russian]{babel}
\usepackage{hyperref,amsthm,amsmath,amssymb,color,multirow,graphicx}
\usepackage{inputenc}
\usepackage{multicol}
\usepackage{wrapfig, float,array}

\newtheorem{thm}{Theorem}[section]

\newtheorem{lem}[thm]{Lemma}
\newtheorem{prop}[thm]{Proposition}
\theoremstyle{definition}
\newtheorem{defn}[thm]{Definition}
\newtheorem{rem}[thm]{Remark}

\numberwithin{equation}{section}

\newtheorem*{thm*}{Теорема}
\newtheorem*{lem*}{Лемма}

\newcommand{\be}{\begin{equation}}

\newcommand{\ee}{\end{equation}}

\begin{document}

\sloppy

\begin{flushright}
\begin{tabular}{l}
{\sf Uzbek  Mathematical}\\
{\sf Journal, 2021,\ No XX, pp.\pageref{firstpage}-\pageref{lastpage}}\\
\end{tabular}
\end{flushright}

\sloppy

\begin{center}
\textbf{\large A discrete-time epidemic SISI model}\\

\textbf{ S.K.Shoyimardonov}
\end{center}

{\small \textbf{Abstract.} We consider a discrete-time epidemic SISI model in case when the population
size is a constant, so the per capita death rate is equal to per capita birth rate.
The evolution operator of this model is a non-linear operator which depends on seven parameters.
Reducing it to a quadratic stochastic operator, we prove uniqueness of interior fixed point
of the operator and study the limit behavior of the trajectory under some conditions to parameters.
\\

\textbf{Keywords:} epidemic, discrete-time, fixed point, simplex, trajectory, limit.

\textbf{MSC (2010):} 37N25, 92D30
\\

\makeatletter
\renewcommand{\@evenhead}{\vbox{\thepage \hfil {\it S.K.Shoyimardonov}   \hrule }}
\renewcommand{\@oddhead}{\vbox{\hfill
{\it A discrete-time epidemic SISI model}\hfill
\thepage \hrule}} \makeatother

\label{firstpage}

\section{Introduction}

\ \ \ \ Classical epidemic models usually assume that either immunity does not exist (the SIS model)
or that experiencing the infection provides permanent or temporary protection against it (the SIR
and SIRS models). In the SIS model a typical individual starts off susceptible, at some stage catches the infection and after an infectious period becomes completely susceptible again. But, in SISI model a member of a population after recovering can infected second time also.
In \cite{Green} the main example of SISI epidemic model was discussed as bovine respiratory syncytial virus (BRSV) amongst cattle in continuous time.
In this paper it was assumed that the population size under consideration is a constant,
so the per capita death rate is equal to per capita birth rate.
In epidemiology, a \emph{susceptible} individual is a member of a population who is at risk of
becoming infected by a disease. A \emph{susceptibility} only refers to the fact that the
virus is able to get into the cell. \emph{Infectivity} is the ability of a pathogen to
establish an infection. A \emph{Pathogen} is any organism that can produce disease.

\ \ \ \ Continuous time SISI model is as follows \cite{Muller}:
\begin{equation}
\begin{cases}
\frac{dS}{dt} & =b(S+I+S_1+I_1)-\mu S-\beta_1 A(I,I_1)S\\[2mm]
\frac{dI}{dt} & =-\mu I-\alpha I+\beta_1 A(I,I_1)S\\[2mm]
\frac{dS_1}{dt} & =-\mu S_1+\alpha I- \beta_2 A(I,I_1)S_1\\[2mm]
\frac{dI_1}{dt} & =-\mu I_1+\beta_2 A(I,I_1)S_1
\end{cases}\label{eq:Eq1}
\end{equation}
where
$S-$  density of susceptibles who did not have the disease before,
$I-$  density of first time infected persons,
$S_1-$  density of recovereds,
$I_1-$  density of second time infected persons,
$b-$ birth rate,
$\mu-$  death rate,
$\alpha-$ recovery rate,
$\beta_1-$  susceptibility of persons in $S$,
$\beta_2-$  susceptibility of persons in $S_1,$
$k_1-$  infectivity of persons in $I$,
$k_2-$  infectivity of persons in $I_1.$
Moreover, $A(I,I_1)$ denotes the so-called force of infection,
$$A(I,I_1)=\frac{k_1I+k_2I_1}{P}$$
and $P=S+I+S_1+I_1$ denotes the total population size.

In \cite{SS} it was assumed that birth rate is a same with death rate, i.e., $b=\mu$ and using some replacements it was obtained a new system:
$$x=\frac{S}{P}, \ \ u=\frac{I}{P}, \ \ y=\frac{S_1}{P}, \ \ v=\frac{I_1}{P},$$

\begin{equation}
\begin{cases}
\frac{dx}{dt} & =b-bx-\beta_1 A(u,v)x\\[2mm]
\frac{du}{dt} & =-bu-\alpha u+\beta_1 A(u,v)x\\[2mm]
\frac{dy}{dt} & =-by+\alpha u -\beta_2 A(u,v)y\\[2mm]
\frac{dv}{dt} & =-bv+\beta_2 A(u,v)y
\end{cases}\label{cont}
\end{equation}
where all parameters are non-negative and $x+u+y+v=1$.

\emph{The quadratic stochastic operator} (QSO) \cite{GMR}, \cite{L}, \cite{Rpd} is
a mapping of the standard simplex.
\begin{equation}
S^{m-1}=\{x=(x_{1},...,x_{m})\in\mathbb{R}^{m}:x_{i}\geq0,\sum\limits _{i=1}^{m}x_{i}=1\}\label{2}
\end{equation}
into itself, of the form
\begin{equation}
V:x'_{k}=\sum\limits _{i=1}^{m}\sum\limits _{j=1}^{m}P_{ij,k}x_{i}x_{j},\qquad k=1,...,m,\label{3}
\end{equation}
where the coefficients $P_{ij,k}$ satisfy the following conditions
\begin{equation}
P_{ij,k}\geq0,\quad P_{ij,k}=P_{ji,k},\quad\sum\limits _{k=1}^{m}P_{ij,k}=1,\qquad(i,j,k=1,...,m).\label{4}
\end{equation}

For a given $\lambda^{(0)}\in S^{m-1}$ the \emph{trajectory}
(orbit) $\{\lambda^{(n)};n\geq0\}$ of $\lambda^{(0)}$ under
the action of QSO (\ref{3}) is defined by
\[
\lambda^{(n+1)}=V(\lambda^{(n)}),\;n=0,1,2,...
\]

The main problem in mathematical biology consists in the study of the asymptotical behaviour of the trajectories.

\section{Main Part}
\ \ \ \  In \cite{SS} it was studied the following discrete-time version of the system
(\ref{cont}).

\begin{equation}
V:\left\{ \begin{alignedat}{1}x^{(1)} & =x+b-bx-\beta_1 A(u,v)x\\
u^{(1)} & =u-bu-\alpha u+\beta_1 A(u,v)x\\
y^{(1)} & =y-by+\alpha u -\beta_2 A(u,v)y\\
v^{(1)} & =v-bv+\beta_2 A(u,v)y
\end{alignedat}
\right.\label{disc}
\end{equation}
where $A(u,v)=k_1u+k_2v.$

\begin{prop}\cite{SS}
We have $V\left(S^{3}\right)\subset S^{3}$ if and only if the
non-negative parameters $b,\alpha, \beta_1, \beta_2, k_1, k_2$
verify the following conditions
\begin{equation}
\begin{array}{cccc}
\alpha+b\leq1, & \beta_{1}k_{2}\leq2, & \beta_{2}k_{1}\leq2,\medskip\\
b+ \beta_{2}k_{2}\leq1, & \left|b-\beta_{1}k_{1}\right|\leq1, & \left|b-\beta_{2}k_{2}\right|\leq1,\medskip\\
\left|b-\beta_{1}k_{2}\right|\leq1,  & \left|\alpha+b-\beta_{1}k_{1}\right|\leq1, & \left|\alpha-b-\beta_{2}k_{1}\right|\leq1.
\end{array}\label{cond}
\end{equation}

Moreover, under conditions (\ref{cond}) the operator $V$ is
a QSO.
\end{prop}

Recall that fixed point of the operator $F$ is a solution of equation $F(x)=x.$

Let us denote
\begin{itemize}
\item[] $\lambda_{1}=\left(1,0,0,0\right), \ \  \lambda_{2}=\left(0,1,0,0\right), \ \ \lambda_{3}=\left(0,0,1,0\right),\ \ \lambda_{4}=\left(0,0,0,1\right),$
\item[] $\Lambda_{5}=\{\lambda=(x,u,y,v)\in S^3: x=u=0\},$
\item[] $\Lambda_{6}=\{\lambda=(x,u,y,v)\in S^3: x=y=0\},$
\item[] $\Lambda_{7}=\{\lambda=(x,u,y,v)\in S^3: x=v=0\},$
\item[] $\Lambda_{8}=\{\lambda=(x,u,y,v)\in S^3: u=y=0\},$
\item[] $\Lambda_{9}=\{\lambda=(x,u,y,v)\in S^3: u=v=0\},$
\item[] $\Lambda_{10}=\{\lambda=(x,u,y,v)\in S^3: y=v=0\},$
\item[] $\Lambda_{11}=\{\lambda=(x,u,y,v)\in S^3: x=0\},$
\item[] $\Lambda_{12}=\{\lambda=(x,u,y,v)\in S^3: u=0\},$
\item[] $\Lambda_{13}=\{\lambda=(x,u,y,v)\in S^3: y=0\},$
\item[] $\Lambda_{14}=\{\lambda=(x,u,y,v)\in S^3: v=0\},$
\item[]  $\lambda_{15}=\left(\frac{b}{\beta_1 k_1},\frac{\beta_1 k_1-b}{\beta_1 k_1},0,0\right), \ \ \lambda_{16}=\left(\frac{b+\alpha}{\beta_1 k_1},\frac{b(\beta_1k_1-b-\alpha)}{\beta_1k_1(b+\alpha)},\frac{\alpha(\beta_1k_1-b-\alpha)}{\beta_1k_1(b+\alpha)},0 \right),$

\item[] $\lambda_{17}=\left(\frac{b}{b+\beta_1 A},\frac{b\beta_1A}{(b+\beta_1 A)(b+\alpha)},\frac{\alpha b\beta_1A}{(b+\beta_1 A)(b+\beta_2A)(b+\alpha)},\frac{\alpha \beta_1\beta_2 A^2}{(b+\beta_1 A)(b+\beta_2A)(b+\alpha)} \right),$

where  $A$ is a positive solution of the equation
\begin{equation}\label{fpc}
1=\frac{b\beta_1k_1}{(b+\beta_1 A)(b+\alpha)}+\frac{\alpha \beta_1\beta_2k_2A}{(b+\beta_1 A)(b+\beta_2A)(b+\alpha)}
\end{equation}
\end{itemize}

\begin{rem} If $b=\alpha=k_1=k_2=0$ or $b=\alpha=\beta_1=\beta_2=0$ then the operator (\ref{disc}) is an identity operator.
\end{rem}
By the following proposition we give all possible fixed points
of the operator $V.$
\begin{prop}\cite{SS}
\label{fp} Let $Fix(V)$ be set of fixed points of the operator (\ref{disc}). Then

$$Fix(V)=\left\{\begin{array}{lll}
\{\lambda_1\} \ \ \\[2mm]
\{\lambda_4\}\bigcup\Lambda_{9}, \ \ {\rm if} \ \  b=0 \\[2mm]
\Lambda_{6}\bigcup\Lambda_{9}, \ \ \ \ {\rm if} \ \ b=\alpha=0 \\[2mm]
\Lambda_{8}\bigcup\Lambda_{9}, \ \ \ \ {\rm if} \ \ b=\beta_1=0 \\[2mm]
\Lambda_{5}\bigcup\Lambda_{9}, \ \ \ \ {\rm if} \ \ b=\beta_2=0 \\[2mm]
\{\lambda_{4}\}\bigcup\Lambda_{9}, \ \ \ \ {\rm if} \ \ b=k_1=0 \\[2mm]
\Lambda_{12}, \ \ \ \ {\rm if} \ \ b=k_2=0 \\[2mm]
\Lambda_{9}\bigcup\Lambda_{13}, \ \ \ \ {\rm if} \ \ b=\alpha=\beta_1=0 \\[2mm]
\Lambda_{9}\bigcup\Lambda_{11}, \ \ \ \ {\rm if} \ \ b=\alpha=\beta_2=0 \\[2mm]
\Lambda_{6}\bigcup\Lambda_{14}, \ \ \ \ {\rm if} \ \ b=\alpha=k_1=0 \\[2mm]
\Lambda_{6}\bigcup\Lambda_{12}, \ \ \ \ {\rm if} \ \ b=\alpha=k_2=0 \\[2mm]
\Lambda_{12}, \ \ \ \ \ \ \ \ \ \ \  {\rm if} \ \ b=\beta_1=\beta_2=0 \\[2mm]
\Lambda_{8}\bigcup\Lambda_{9}, \ \ \ \ {\rm if} \ \ b=\beta_1=k_1=0 \\[2mm]
\Lambda_{12}, \ \ \ \ {\rm if} \ \ b=\beta_1=k_2=0 \\[2mm]
\Lambda_{5}\bigcup\Lambda_{9}, \ \ \ \ {\rm if} \ \ b=\beta_2=k_1=0 \\[2mm]
\Lambda_{12},  \ \ \ \ {\rm if} \ \ b=\beta_2=k_2=0 \\[2mm]
\{\lambda_1, \lambda_{15} \}, \ \ \ \ {\rm if} \ \ b>0, \alpha=0 \ \ {\rm and} \ \ \beta_1k_1>b \\[2mm]
\{\lambda_1, \lambda_{16} \}, \ \ \ \ {\rm if} \ \ b>0,\, \alpha>0,\, \beta_2=0, \, \beta_1 k_1>b+\alpha\\[2mm]
\{\lambda_1, \lambda_{17} \}, \ \ \ \ {\rm if} \ \ \alpha b \beta_1\beta_2 k_1>0
\end{array}\right.$$
\end{prop}

We interested in the positive solutions of the equation (\ref{fpc}).

\begin{thm} For the equation (\ref{fpc}) the following cases hold:\\
(i) If $\beta_1k_1= b+\alpha$ and $\alpha\beta_2k_2>b\beta_1k_1$ then the equation (\ref{fpc}) has unique positive solution;\\
(ii) If $\beta_1k_1> b+\alpha$ the the equation (\ref{fpc}) has unique positive solution;\\
(iii) If $\beta_1k_1<b+\alpha$ then the equation (\ref{fpc}) does not have positive solution.
\end{thm}
\begin{proof} First of all, we denote by $f(x), g(x):$
\begin{equation}\label{eqq}
f(x)=b+\beta_1x, \ \ \ \ \ \  g(x)=\frac{b\beta_1k_1}{b+\alpha}+\frac{\alpha\beta_1\beta_2k_2x}{(b+\beta_2x)(b+\alpha)}.
\end{equation}
Then the roots of the equation (\ref{fpc})  are roots of the equation $f(x)=g(x).$\\

\emph{(i)} If  $\beta_1k_1= b+\alpha$ then by (\ref{eqq}) it follows that
$$\beta_1x=\frac{\alpha\beta_2k_2x}{k_1(b+\beta_2x)} \ \ \Rightarrow x_1=0, \ \ x_2=\frac{\alpha\beta_2k_2-b\beta_1k_1}{\beta_1\beta_2k_1}.$$
By condition of the theorem  $x_2>0.$\\

\emph{(ii)}  First, we find the horizontal asymptote of the graph of the function $g(x):$
 $$y=\frac{\beta_1(bk_1+\alpha k_2)}{b+\alpha}=const.$$
 Moreover, right part (with respect to vertical asymptote) of the graphic of $g(x)$ is increasing and convex.   By the condition of theorem $\beta_1k_1> b+\alpha,$ so $\frac{b\beta_1k_1}{b+\alpha}>b,$ i.e., $f(0)<g(0).$ Therefore, the equation (\ref{eqq}) has one positive solution (Fig. \ref{fig1}).\\

 \emph{(iii)}  Let $\beta_1k_1<b+\alpha$. Then $f(0)>g(0).$ Assume that $f(x)=g(x)$ has positive solution. Then there are two solutions $x_1, x_2$ (suppose that $x_1<x_2$) and for any $x\in(x_1;x_2),$ $f(x)<g(x).$ Note that $g(x)$ is convex and increasing function, so there exists $\bar{x}\in(x_1;x_2)$ such that at this point $f(x)$ and $g(x)$ have same slope. Let's equate these slopes at the point $\bar{x}$:
 $$\beta_1=\frac{\alpha b\beta_1\beta_2k_2}{(b+\alpha)(b+\beta_2\bar{x})^2} \Rightarrow \beta_2\bar{x}=-b+\sqrt{\frac{\alpha b\beta_2k_2}{b+\alpha}}.$$
 Since $\bar{x}>0$ one obtains that $k_2>\frac{b(b+\alpha)}{\alpha \beta_2}.$ In addition, $f(\bar{x})<g(\bar{x}),$ from this after some calculations we get the condition for parameters:
 $$k_2>\frac{b(b+\alpha)}{\alpha \beta_2}\left(\frac{\beta_2}{\beta_1}\left(1-\frac{\beta_1k_1}{b+\alpha}\right)+\frac{1}{b}\right)^2$$
But, from $\beta_1k_1<b+\alpha$ and using positiveness all parameters we have
$$k_2>\frac{b(b+\alpha)}{\alpha \beta_2}\left(\frac{\beta_2}{\beta_1}\left(1-\frac{\beta_1k_1}{b+\alpha}\right)+\frac{1}{b}\right)^2>\frac{b(b+\alpha)}{\alpha \beta_2}\left(\frac{1}{b}\right)^2=\frac{b+\alpha}{b\alpha\beta_2}>\frac{1}{\alpha\beta_2}.$$
From this we get $\beta_2k_2>\frac{1}{\alpha}>1$ which is contradiction to the conditions (\ref{cond}). Hence, there is no any positive solution in this case (Fig. \ref{fig2}). Theorem is proved.
\end{proof}

\begin{figure}[h!]
\begin{multicols}{2}
\hfill
\includegraphics[width=5.8cm]{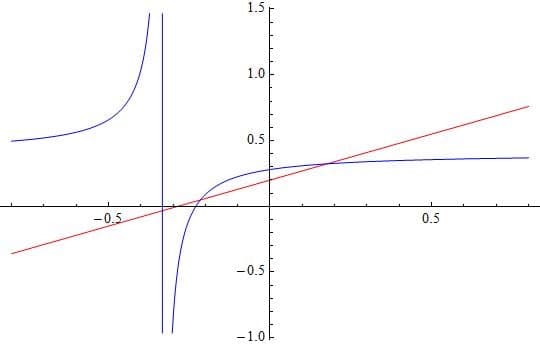}
\hfill
\caption{ $\beta_1k_1>b+\alpha: \alpha=0.3, b=0.2, \beta_1=0.7, \beta_2=0.6, k_1=1, k_2=0.3$}\label{fig1}
\label{figLeft}
\hfill
\includegraphics[width=5.9cm]{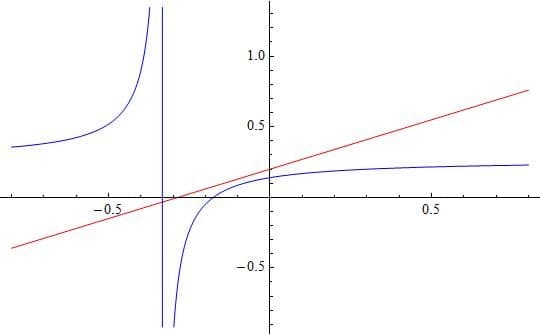}
\hfill
\caption{ $\beta_1k_1<b+\alpha: \alpha=0.3, b=0.2, \beta_1=0.7, \beta_2=0.6, k_1=0.5, k_2=0.3$}\label{fig2}
\label{figRight}
\end{multicols}
\end{figure}

\begin{defn}
\cite{De}. A fixed point $p$ for $F:\mathbb{R}^{m}\rightarrow\mathbb{R}^{m}$
is called \emph{hyperbolic} if the Jacobian matrix $\textbf{J}=\textbf{J}_{F}$
of the map $F$ at the point $p$ has no eigenvalues on the unit
circle.There are three types of hyperbolic fixed points:

(1) $p$ is an attracting fixed point if all of the eigenvalues
of $\textbf{J}(p)$ are less than one in absolute value.

(2) $p$ is an repelling fixed point if all of the eigenvalues
of $\textbf{J}(p)$ are greater than one in absolute value.

(3) $p$ is a saddle point otherwise.
\end{defn}

\textbf{Conjecture 1.} \cite{SS} If $\beta_2=0$ then for an initial point $\lambda^{0}=\left(x^{0},u^{0},y^{0},v^{0}\right)\in S^{3}$
(except fixed points) the trajectory has the following limit
\[
\lim_{n\to\infty}V^{(n)}(\lambda^{0})=\begin{cases}
\lambda_1  & \text{if } \ \  \ \ \beta_1k_1\leq b+\alpha, b\alpha>0 \\
\lambda_{16} & \text{if } \ \ u^0+v^0>0 \ \ \text{and} \ \ \beta_1k_1>b+\alpha, b\alpha>0\\
\end{cases}
\]

\begin{thm} If $\beta_2=0$ and $\beta_1k_1>b+\alpha, b\alpha>0$ then there exists a neighborhood $U(\lambda_{16})$
such that
$$\lim_{n\to \infty}V^{(n)}=\lambda_{16}.$$
\end{thm}

\begin{proof} If we prove the attractiveness of the fixed point $\lambda_{16}$ then based on general theory (see \cite{De}, \cite{Rmb}) we can say that there exists required neighbourhood. Assume that $\beta_2=0$ then the Jacobian matrix of the operator $V$ is
\[ J=
\left[\begin{array}{cccccc}
1-b-\beta_1A & -\beta_1k_1x &0 & -\beta_1k_2x\\
\beta_1A & 1-b-\alpha+\beta_1k_1x & 0 & \beta_1k_2x \\
0 & \alpha & 1-b& 0 \\
0 & 0 & 0& 1-b
\end{array}\right]
\]

Recall that $\lambda_{16}=\left(\frac{b+\alpha}{\beta_1 k_1},\frac{b(\beta_1k_1-b-\alpha)}{\beta_1k_1(b+\alpha)},\frac{\alpha(\beta_1k_1-b-\alpha)}{\beta_1k_1(b+\alpha)},0 \right),$ then

\[ J(\lambda_{16})=
\left[\begin{array}{cccccc}
1-b-\beta_1A & -(b+\alpha) &0 & -\frac{k_2}{k_1}(b+\alpha)\\
\beta_1A & 1 & 0 & \frac{k_2}{k_1}(b+\alpha) \\
0 & \alpha & 1-b& 0 \\
0 & 0 & 0& 1-b
\end{array}\right]
\]
From this to find eigenvalues we get

\[ |J(\lambda_{16})-\mu E|=
\left|\begin{array}{cccccc}
1-b-\beta_1A-\mu & -(b+\alpha) &0 & -\frac{k_2}{k_1}(b+\alpha)\\
\beta_1A & 1-\mu & 0 & \frac{k_2}{k_1}(b+\alpha) \\
0 & \alpha & 1-b-\mu& 0 \\
0 & 0 & 0& 1-b-\mu
\end{array}\right|=0
\]
where $E$ is an identity matrix, solving this determinant we obtain the following equation:
\begin{equation}\label{char}
(1-b-\mu)^2((1-b-\beta_1A-\mu)(1-\mu)+\beta_1A(b+\alpha))=0
\end{equation}
Thus, two coinciding solutions of the (\ref{char}) are $\mu_1=\mu_2=1-b$ which are always between 0 and 1.  Next, we find other solutions, from (\ref{char}) we have

$$\mu^2-(2-b-\beta_1A)\mu+\beta_1A(b+\alpha)+1-b-\beta_1A=0 \ \ \ \ \Rightarrow $$

$$\mu^2-(1-b+1-\beta_1A)\mu+\beta_1A\alpha+(1-b)(1-\beta_1A)=0$$
then the discriminant is
$$D=(1-b+1-\beta_1A)^2-4(1-b)(1-\beta_1A)-4\beta_1A\alpha=(b-\beta_1A)^2-4\beta_1A\alpha$$
thus,
$$\mu_3=\frac{2-b-\beta_1A-\sqrt{(b-\beta_1A)^2-4\beta_1A\alpha}}{2}, $$ $$\mu_4=\frac{2-b-\beta_1A+\sqrt{(b-\beta_1A)^2-4\beta_1A\alpha}}{2}. $$
where $A=k_1u=\frac{b(\beta_1k_1-b-\alpha)}{\beta_1(b+\alpha)}.$

\textbf{Case:} $(b-\beta_1A)^2-4\beta_1A\alpha\geq0.$ Let us show that $|\mu_3|<1, |\mu_4|<1.$ 
Since $\beta_1k_1>b+\alpha, b\alpha>0$ we have
$$b+\beta_1A+\sqrt{(b-\beta_1A)^2-4\beta_1A\alpha}>0$$ and $$b+\beta_1A-\sqrt{(b-\beta_1A)^2-4\beta_1A\alpha}>b+\beta_1A-\sqrt{(b-\beta_1A)^2}>b+\beta_1A-|b-\beta_1A|>0.$$
From this we get $\mu_3<1, \mu_4<1.$
We consider the conditions $\mu_3>-1, \mu_4>-1.$ It is enough to show that
$$\frac{2-b-\beta_1A-\sqrt{(b-\beta_1A)^2-4\beta_1A\alpha}}{2}>-1,$$
 i.e.,  $b+\beta_1A+\sqrt{(b-\beta_1A)^2-4\beta_1A\alpha}<4.$
Since
$$b+\beta_1A+\sqrt{(b-\beta_1A)^2-4\beta_1A\alpha}<b+\beta_1A+|b-\beta_1A|,$$
it follows that if $b>\beta_1A$ then $2b<2<4,$ if $b<\beta_1A$ then $2\beta_1A=2\beta_1k_1u<2\beta_1k_1\leq2(1+b)<4,$ because, by the conditions (\ref{cond}) it obtains $|b-\beta_{1}k_{1}|\leq1 \Rightarrow \beta_{1}k_{1}\leq1+b.$

\textbf{Case:} $(b-\beta_1A)^2-4\beta_1A\alpha<0.$ First, we show that $\beta_1A\leq1.$ By the conditions (\ref{cond}) we have $|\alpha+b-\beta_{1}k_{1}|\leq1$ and from this $\beta_{1}k_{1}-b-\alpha\leq1,$ so $\beta_1A=\frac{b(\beta_1k_1-b-\alpha)}{b+\alpha}<1.$ In addition, $\mu_{3,4}=1-\frac{b+\beta_1A}{2}\mp\frac{\sqrt{4\beta_1A\alpha-(b-\beta_1A)^2}}{2}i,$ and from this

$$|\mu_{3,4}|=\sqrt{\left(1-\frac{b+\beta_1A}{2}\right)^2+\frac{4\beta_1A\alpha-(b-\beta_1A)^2}{4}}=$$\\
$=\sqrt{1-\beta_1A(1-\alpha)-b(1-\beta_1A)}<1.$

The proof of the Theorem is completed.
\end{proof}

 \textbf{Conjecture 2.}\cite{SS} If $\alpha b \beta_1\beta_2 k_1>0$ then for an initial point $\lambda^{0}=\left(x^{0},u^{0},y^{0},v^{0}\right)\in S^{3}$
(except fixed points) the trajectory  has the following limit
\[
\lim_{n\to\infty}V^{(n)}(\lambda^{0})=\begin{cases}
\lambda_1  & \text{if } \ \ u^0=v^0=0 \ \ \text{or} \ \ \beta_1k_1\leq b+\alpha \\
\lambda_{17} & \text{if } \ \  u^0+v^0>0 \ \ \text{and} \ \ \beta_1k_1>b+\alpha\\
\end{cases}
\]

\begin{thm} If $k_2=0$ then for an initial point $\lambda^{0}=\left(x^{0},u^{0},y^{0},v^{0}\right)\in S^{3}$
(except fixed points) the trajectory  has the following limit
\[
\lim_{n\to\infty}V^{(n)}(\lambda^{0})=\lambda_1=(1;0;0;0)  \ \ if \ \ u^0=0 \ \ or \ \ \beta_1k_1\leq b+\alpha
\]
\end{thm}
\begin{proof} Suppose that $k_2=0$ then the operator (\ref{disc}) has the following representation:
\begin{equation}
V:\left\{ \begin{alignedat}{1}x^{(1)} & =x+b-bx-\beta_1k_1ux\\
u^{(1)} & =u-bu-\alpha u+\beta_1k_1ux\\
y^{(1)} & =y-by+\alpha u -\beta_2k_1uy\\
v^{(1)} & =v-bv+\beta_2k_1uy
\end{alignedat}
\right.\label{thm3}
\end{equation}

1) Assume $u^0=0$ then $u^{(n)}=0,$ and so from first equation of (\ref{thm3}) we have
 \[
 \lim_{n\to\infty}x^{(n+1)}=x\lim_{n\to\infty}(1-b)^n+b\sum_{k=0}^{\infty}(1-b)^k=1 \ \  \text{for any} \ \ \lambda^0\in S^3
\]
and a limit of the operator is $\lambda_1.$

2) For $\beta_1k_1\leq b+\alpha$ from second equation of (\ref{thm3}) one has $u^{(1)}=u-(b+\alpha-\beta_1k_1x)u\leq u,$ so the sequence $u^{(n)}$ has a limit $\bar{u}.$ Assume $\bar{u}>0,$ then by taking a limit we get
\[
\lim_{n\to\infty}u^{(n+1)}=\lim_{n\to\infty}u^{(n)}-\lim_{n\to\infty}(b+\alpha-\beta_1k_1x^{(n)})u^{(n)} \Rightarrow
\]
\[
\Rightarrow \bar{u}=\bar{u}-(b+\alpha)\bar{u}+\beta_1k_1\lim_{n\to\infty}x^{(n)}u^{(n)} \Rightarrow \lim_{n\to\infty}x^{(n)}\left(\frac{u^{(n)}}{\bar{u}}\right)=\frac{b+\alpha}{\beta_1k_1}\geq1
\]
Note that $u^{(n)}$ non-increasing sequence and for any $n\in N,$ $x^{(n)}\leq1,$ so from $\frac{b+\alpha}{\beta_1k_1}\geq1$ we have a contradiction. Thus, $\bar{u}=0.$ Next, by adding last two equations of (\ref{thm3}) we have $y^{(1)}+v^{(1)}=(1-b)(y+v)+\alpha u.$ By denoting $y^{(n)}+v^{(n)}=z^{(n)}$ one has $z^{(1)}=(1-b)z+\alpha u.$
We formulate a new operator with respect to $u,z:$
\begin{equation}
W:\left\{ \begin{alignedat}{1}u^{(1)} & =u-bu-\alpha u+\beta_1k_1ux\\
z^{(1)} & =z-bz+\alpha u\\
\end{alignedat}
\right.\label{newoper}
\end{equation}
 Here we have a useful lemma.

\begin{lem}\label{lem1} If $\beta_1k_1\leq b+\alpha$ then the set $$M=\{(u;y)\in S^1: bz-\alpha u\geq0\}$$
is an invariant with respect to operator (\ref{newoper}).
\end{lem}
\begin{proof} Let $(u,z)\in M,$ i.e., $bz-\alpha u\geq0.$ We check the condition $bz^{(1)}-\alpha u^{(1)}\geq0$:
$$
bz^{(1)}-\alpha u^{(1)}=b(z-bz+\alpha u)-\alpha(u-bu-\alpha u+\beta_1k_1ux)=$$
$$=bz-\alpha u+b\alpha u-b^2z+b\alpha u+\alpha^2u-\alpha \beta_1k_1ux=$$
$$=bz-\alpha u-b(bz-\alpha u)+\alpha u(b+\alpha-\beta_1k_1x)=$$
$$=(bz-\alpha u)(1-b)+\alpha u(b+\alpha-\beta_1k_1x)\geq0.
$$
Thus, the Lemma is proved.
\end{proof}
By the Lemma (\ref{lem1}) we show the existence of limit $z^{(n)}.$ Assume that $(u^0,z^0)\in M$ then $z^{(n)}$ is decreasing. Let be $(u^0,z^0)\notin M$ then there are two possible cases: \\
 (a) If after some finite step $k,$ $(u^{(k)},z^{(k)})\in M$ then $(u^{(n)},z^{(n)})\in M$ for all $n>k,$ so $z^{(n)}$ is decreasing for all $n>k.$\\
 (b) If $(u^{(n)},z^{(n)})\notin M$ for any $n\in N$ then $bz^{(n)}-\alpha u^{(n)}<0,$ $\forall n\in N,$  and so the sequence $z^{(n)}$ is increasing.
Therefore, the sequences $z^{(n)}$ has a limit $\bar{z}.$ If we take a limit from $z^{(n+1)}=(1-b)z^{(n)}+\alpha u^{(n)},$ from positiveness of $b$ and from $\bar{u}=0$ we get $\bar{z}=0.$ Consequently,
 $$\lim_{n\to\infty}(u^{(n)}+y^{(n)}+v^{(n)})=0$$
 and so $\lim_{n\to\infty}x^{(n)}=0.$ The proof of the theorem is finished.
 \end{proof}

\section{Conclusion}
The crucial point of the SISI model is that an individual can be infected twice. We proved uniqueness of the positive solution of the equation (\ref{fpc}) and it provides uniqueness of fixed point of the operator  inside of the simplex $S^3.$ Biological meaning of this result is that if susceptibility of persons in $S$ multiplying infectivity of persons in $I$ greater than sum of birth rate and recovery rate then the disease becomes endemic at the limit (assume that operator has a limit). In addition, there is no susceptibility of persons in $S_1$ such that for any initial state near $\lambda_{16}$ number of second time infected persons disappears at the limit.  If  susceptibility of persons in $S$ multiplying infectivity of persons in $I$ less than sum of birth rate and recovery rate and either there is no infectivity of persons in $I_1$ or susceptibility of persons in $S_1$  then for any initial states stays only susceptible persons who did not have the disease before.
 \hfill$\Box$ \

{\small
\begin{tabular}{p{9cm}}
    Shoyimardonov S.K.,\\
    V.I.Romanovskiy Institute of Mathematics, 4, University str., Tashkent, Uzbekistan. \\
    e-mail: shoyimardonov@inbox.ru\\
    \\

\end{tabular}
}

\label{lastpage}


\textbf{References}
\begin{enumerate}
\bibitem{Muller} Johannes M$\ddot{u}$ller, Christina Kuttler. Methods and models in mathematical biology. Springer, 2015, 721 p.
\bibitem{Green} D. Greenhalgh, O. Diekmann, M. de Jong. Subcritical endemic steady states in mathematical
models for animal infections with incomplete immunity. Math.Biosc. 2000, 165, 25 p.
\bibitem{SS} S.K. Shoyimardonov. A non-linear discrete-time dynamical system related to epidemic SISI model.
In arXiv:2008.12677v1 [math.DS], 2020, 17 p.
\bibitem{De} R.L. Devaney. An Introduction to Chaotic Dynamical System. Westview Press. 2003, 336 p.
\bibitem{GMR} R.N. Ganikhodzhaev, F.M. Mukhamedov, U.A. Rozikov. Quadratic stochastic operators and processes: results and open
problems, Inf. Dim. Anal. Quant. Prob. Rel. Fields. \textbf{14}(2). 2011, pp. 279-335.
\bibitem{L} Y.I. Lyubich. Mathematical structures in population genetics, Springer-Verlag, Berlin. 1992, 383 p.
\bibitem{Rpd} U.A. Rozikov,  Population dynamics: algebraic and probabilistic approach. {\sl World Sci. Publ}. Singapore. 2020, 460 pp.
\bibitem{Rmb} U.A. Rozikov, An introduction to mathematical billiards. {\sl World Sci. Publ}. Singapore. 2019, 224 pp.
\end{enumerate}
\end{document}